\documentclass[12pt]{article}
\usepackage{amsmath,amsthm,amssymb,color,amscd}

\def\seq#1#2#3{#1_{#2},\,\ldots,#1_{#3}}

\def\w{\widetilde}

\def\b{\overline}

\def\vv{{\underline{v}}}

\def\nunu{\underline{\nu}}
\def\tt{{\underline{t}}}

\def\mm{\underline{m}}
\def\kk{\underline{k}}

\def\1{\underline{1}}
\def\0{\underline{0}}
\def\R{\mathbb R}
\def\P{\mathbb P}

\def\LL{{\cal L}}

\def\Z{\mathbb Z}

\def\C{\mathbb C}
\def\K{\mathbb K}

\def\OO{{\mathcal O}}
\def\EE{{\mathcal E}}

\newtheorem{theorem}{Theorem}

\newtheorem{lemma}{Lemma}
\newtheorem{proposition}{Proposition}

\theoremstyle{definition}

\newtheorem{remark}{Remark}

\title{On real analogues of the Poincar\'e series
}

\author{
A.~Campillo,
\and F.~Delgado, 
\and S.M.~Gusein-Zade\thanks{
The work of the third author
is an output of a research project implemented as part of the Basic Research Program at the National Research University Higher School of Economics (HSE University).
} }

\date{}
\begin{document}

\def\eps{\varepsilon}

\maketitle

\begin{abstract}
There exist several equivalent equations for the Poincar\'e series of a collection of valuations on the ring of germs of functions on a complex analytic variety. We give
definitions of the Poinca\'e series of a collection of valuations in the real setting (i.\,e., on the ring of germs of functions on a real analytic variety), compute them for the case of one curve valuation on the plane and discuss some of their properties.
\end{abstract}

\section{Introduction}\label{sec:Introduction}
Let $\K$ be either the field $\C$ of complex numbers
or the field $\R$ of real numbers, let $(V,0)$ be a germ of an analytic variety over $\K$, and let $\EE_{V,0}$ be the
ring of germs of functions on $V$ (if $\K=\C$, $\EE_{V,0}$
is usually denoted by $\OO_{V,0}$).
A function
$\nu:\EE_{V,0}\to\Z_{\ge 0}\cup{\{+\infty\}}$
is called an {\em order function} on $\EE_{V,0}$ if
\begin{enumerate}
\item[1)] $\nu(0)=+\infty$;
\item[2)] $\nu(\lambda f)=\nu(f)$ for $f\in \EE_{V,0}$,
$\lambda\in\K^*:=\K\setminus\{0\}$;
\item[3)] $\nu(f_1+f_2)\ge\min\{\nu(f_1),\nu(f_2)\}$ for $f_1,f_2\in \EE_{V,0}$.
\end{enumerate}
If besides that one has
$$
\nu(f_1f_2)=\nu(f_1)+\nu(f_2)\,,
$$
the function $\nu$ is called a {\em valuation}. (Sometimes for a valuation one demand that $\nu(f)\ne+\infty$ for $f\ne 0$.
In this case a function described above is called a pre-valuation.)

Let $\{\nu_i\}$, $i=1, 2, \ldots, r$, be a collection of order functions on $\EE_{V,0}$. The Poincar\'e series of the collection $\{\nu_i\}$ was defined in~\cite{CDK}.
For $\mm=(m_1,\ldots,m_r)$ and $\mm'=(m_1',\ldots,m_r')$
from $\Z^r$, one says that $\mm\ge\mm'$ if $m_i\ge m_i'$ for all $i$.
For $f\in\EE_{V,0}$, let $\nunu(f):=(\nu_1(f),\ldots,\nu_r(f))$.
For $\kk\in\Z^r$, let $J_{\kk}:=\{f\in\EE_{V,0}:\nunu(f)\ge\kk\}$.
Let
$$
\LL_{\{\nu_i\}}(\tt):=\sum_{\kk\in\Z^r}
\dim\left({J_{\kk}/J_{\kk+\1}}\right)\cdot\tt^{\kk},
$$
where $\tt=(t_1,\ldots,t_r)$,
$\1=(1, \ldots, 1)$.
(Pay attention that the sum is over all $\kk\in\Z^r$ and therefore the series $\LL_{\{\nu_i\}}(\tt)$ contains
summands with negative exponents (at least for $r>1$).)
The {\em Poincar\'e series} of the collection $\{\nu_i\}$
is defined by
\begin{equation}\label{eqn:Poincare_ini}
P_{\{\nu_i\}}(\tt)=\frac{\LL_{\{\nu_i\}}(\tt)\cdot\prod_{i=1}^r(t_i-1)}
{(t_1\cdot\ldots\cdot t_r-1)}\,.
\end{equation}

For one curve valuation $\nu=\nu_C$ on a complex analytic variety $(V,0)$ ($\K=\C$, $(C,0)$ is a curve germ on $(V,0)$) this definition gives the following. Let $S_C\subset\Z_{\ge 0}$ be the set (a semigroup) of values of the valuation $\nu$. Then
\begin{equation}\label{eqn:semigroup_one_curve}
 P_{\{\nu\}}(t)=\sum_{v\in S_C}t^v.
\end{equation}

In~\cite{Duke}, it was shown that, for $\K=\C$
and for a collection $\{\nu_i\}$ of curve valuations on the algebra $\OO_{\C^2,0}$ of function germs in two variables
defined by the irreducible components $(C_i,0)$ of a plane curve germ $(C,0)$, the Poincar\'e series $P_{\{\nu_i\}}(\tt)$ coincides with the Alexander polynomial (in several
variables) of the link $C\cap S_{\varepsilon}^3\subset S_{\varepsilon}^3$. In particular, for an embedded resolution
$\pi:(X,D)\to(\C^2,0)$ of the curve $C$, one has an
``A'Campo type'' representation
\begin{equation}\label{eqn:A'Campo}
P_{\{\nu_i\}}(\tt)=
\prod_{\sigma}(1-\tt^{\mm_{\sigma}})^{-\chi(\stackrel{\circ}{E}_{\sigma})},
\end{equation}
where $\mm_{\sigma}\in\Z_{>0}^r$ are some multiplicities defined for the
components $E_{\sigma}$ of the exceptional divisor
$D=\pi^{-1}(0)$ and $\stackrel{\circ}{E}_{\sigma}$ is
the ``smooth part'' of the component $E_{\sigma}$, that is
$E_{\sigma}$ itself minus the intersection points with
all other components of the total transform $\pi^{-1}(C)$
of the curve $C$.

The computation of the Poincar\'e series in~\cite{Duke}
was based on its relation with the so-called extended
semigroup of a collection of (curve) valuations.
For a collection $\{\nu_i\}$ of valuations on the algebra
$\EE_{V,0}$ its {\em extended semigroup}
$\widehat{S}_{\{\nu_i\}}$ is the
union over $\vv\in\Z_{\ge0}^r$ of the {\em fibres}
$$
F_{\vv}:=\left(J(\vv)/J(\vv+\1)\right)\setminus
\bigcup_{I\subset I_0,\,I\ne \emptyset}
\left(J(\vv+\1_I)/J(\vv+\1)\right)\,,
$$
where $I_0=\{1,\ldots,r\}$, the $i$th component of
$\1_I\in\Z^r$ is equal to $1$ if $i\in I$ and to $0$
otherwise. The semigroup structure on $\widehat{S}_{\{\nu_i\}}$ is induced by the product of function germs.
In~\cite{Duke}, it was shown that (for $\K=\C$) the Poincar\'e series of the collection $\{\nu_i\}$ is given by the equation
\begin{equation}\label{eqn:Poincare_through_semigroup}
 P_{\{\nu_i\}}(\tt)=
 \sum_{\kk\in\Z_{\ge 0}^r}\chi(\P F_{\kk})\cdot \tt^{\kk},
\end{equation}
where $\P F_{\kk}$ is the projectivization $F_{\kk}/\C^*$
of the fibre of the extended semigroup.
(Here we always have in mind the {\em additive Euler characteristic} defined as the
alternating sum of the ranks of the cohomology groups with compact support. In fact, for
a space being a complex quasiprojective variety (in particular for the projectivizations of the fibres of the extended semigroup in the complex case: $\K=\C$),
this Euler characteristic coincides with the ``usual'' one.)

In~\cite{RMS} (also for $\K=\C$), the Poincar\'e series $P_{\{\nu_i\}}(\tt)$
was expressed in terms of the integral with respect to the
Euler characteristic (appropriately defined) over the projectivisation of ring $\OO_{V,0}$:
\begin{equation}\label{eqn:Poincare_as_integral}
 P_{\{\nu_i\}}(\tt)=\int_{\P\OO_{V,0}}\tt^{\,\nunu(f)}d\chi\,.
\end{equation}
This equation led to a new, relatively simple computation of
the Poincar\'e series in a number of situations.

In the complex setting all three equations of the Poincare series:~(\ref{eqn:Poincare_ini}), (\ref{eqn:Poincare_through_semigroup}), and (\ref{eqn:Poincare_as_integral}) (and the equation~(\ref{eqn:semigroup_one_curve}) for the case of one curve valuation) --- are equivalent to each other. For collections of curve or divisorial valuations on $\OO_{\C^2,0}$
the Poincar\'e series has a cyclotomic form, i.\,e., is product|ratio of powers of binomials of the form $(1-\tt^{\mm})$ like in (\ref{eqn:A'Campo}). Here we give
analogues of definitions of the Poinca\'e series of a collection of valuations (or of order functions) in the real setting based on these equations. The definitions appear to be different.
For the case of one curve valuation on the plane we compute
three of these Poincar\'e series. (All three of them appear
to be cyclotomic.)

\section{Analogues of the Poincar\'e series in the real setting}\label{sec:Definitions}
Let $(V,0)$ be a germ of a real analytic variety,
let $(V_\C,0)$ be its complexification and let $\{\nu_i\}$,
$i=1,\ldots, r$, be a collection of valuations (or of
order functions) on the ring $\OO_{V_\C,0}$
of germs of complex analytic functions on $(V_\C,0)$.
The valuations $\nu_i$ define valuations (denoted in the same way) on the ring $\EE_{V_,0}$
of germs of real analytic functions on $(V,0)$.
(Each valuation on $\EE_{V_,0}$ is the restriction of a valuation on
$\OO_{V_\C,0}$: see, e.g., \cite[Chapter 4, Theorem 1]{Rib}.)
Taking into account Equations~(\ref{eqn:Poincare_ini}), (\ref{eqn:Poincare_through_semigroup}),
and (\ref{eqn:Poincare_as_integral}) for the Poincar\'e series in the complex setting, one can consider some (different) analogues of it
in the described situation.

\begin{enumerate}
 \item[1)] Let $P_{\{\nu_i\}}(\tt)$ be the {\em classical Poincar\'e series} defined by Equation~(\ref{eqn:Poincare_ini}),
 where $J(\vv)=\{f\in\EE_{V,0}: \nunu(f)\ge\vv\}$.
 The space $J(\vv)/J(\vv+\1)$
 is a (real) subspace of the complex vector space
 $J_{\C}(\vv)/J_{\C}(\vv+\1)$,
 where $J_{\C}(\vv)=\{f\in\OO_{V_{\C},0}: \nunu(f)\ge\vv\}$.
 Therefore it is the direct sum of
 a complex vector subspace of
 $J_{\C}(\vv)/J_{\C}(\vv+\1)$
 and of a ``purely real'' one.
 In the case of one curve valuation $\nu$ on $\OO_{V_\C,0}$
 all the coefficients of the series $P_\nu(t)$ are equal to
 $0$, $1$, or $2$ (if the subspace $J(\vv)/J(\vv+\1)$ is
 zero-dimensional, one-dimensional over $\R$, or one-dimensional
 over $\C$ respectively).

 \item[2)] Let the {\em real Poincar\'e series}
 $P^{\R}_{\{\nu_i\}}(\tt)$ be defined as the integral with
 respect to the Euler characteristic of $\tt^\vv$ over the
 projectivization (i.\,e., the quotient by $\R^*$) of the extended semigroup:
 Equation~(\ref{eqn:Poincare_through_semigroup}). In the case of one curve valuation on $\OO_{V_\C,0}$
 all the coefficients of $P^{\R}_\nu(t)$ are equal to
 $0$ or $1$. (One gets this series from $P_{\nu}(t)$
 substituting all the monomials with the coefficient $2$
 by $0$ (equal to the Euler characteristic of $\R\P^1$).

 \item[3)] One has an analogue version defined
 as the integral with
 respect to the Euler characteristic (appropriartelly defined) of $\tt^{\,\nunu(f)}$
over the
 projectivization of the algebra $\EE_{V,0}$:
 Equation~(\ref{eqn:Poincare_as_integral}).
 One can give an equation for it in the spirit
 of~\cite[Theorem~1]{Monats}, however the right hand side
 of it seems to be not really computable. (At least even
 in the simplest cases the result looks very envolved.)
 Therefore we shall not discuss it below (for one curve
 valuation on $\OO_{\C^2,0}$).

 \item[4)] For the case of one valuation one has the
 {\em semigroup Poincar\'e series} $P^{S}_{\nu}(t)$
 defined as the generating series of the semigroup $S_\nu$ of values of $\nu$ (an
analogue of Equation~(\ref{eqn:semigroup_one_curve})):
 $$
 P^{S}_{\nu}(t) =\sum_{v\in S_\nu} t^v.
 $$
 All the coefficients of $P^{S}_\nu(t)$ are equal to
 $0$ or $1$. (In the case of one curve valuation on $\OO_{V_\C,0}$,
 it can be obtained from $P_\nu(t)$
 substituting all the monomials of the form $2t^v$
 by the monomials $t^v$.) A reasonable analogue of this
 Poincar\'e series for the case of several valuations is
 not clear.
\end{enumerate}

\section{The Poincar\'e series for a one curve valuation}
\label{sec:one_valuation}
Let $(C,0)\subset(\C^2,0)$ be an irreducible plane curve germ (where we regard the complex plane $\C^2$ as the
complexification of the real plane $\R^2$) and let
$\pi:(X,D)\to(\C^2,0)$ be the minimal {\em real resolution}
of the curve $C$ (or of the valuation $\nu_C$). This means that it is the minimal resolution of the curve $C\cup \overline{C}$, where $\overline{C}$
is the complex conjugate of the curve $C$. This resolution can be obtained by a sequence of blow ups either at real points or at pairs of complex conjugate points.

\begin{figure}[h]
$$
\unitlength=0.50mm
\begin{picture}(160.00,110.00)(0,-20)
\thinlines
\put(-10,30){\line(1,0){16}} 
\put(8,30){\circle*{0.5}}
\put(10,30){\circle*{0.5}}
\put(12,30){\circle*{0.5}}
\put(14,30){\circle*{0.5}}
\put(16,30){\circle*{0.5}}
\put(18,30){\line(1,0){22}} 
\put(40,30){\circle*{2}} 
\put(37,33){{\scriptsize$\rho$}} 

\put(40,30){\line(2,1){70}} 
\put(110,65){\circle*{2}} 
\put(70,45){\circle*{2}} 
\put(90,55){\circle*{2}} 
\put(110,65){\line(-1,2){10}} 
\put(100,85){\circle*{2}} 
\put(102,87){{\scriptsize$\overline{\sigma_{g}}$}} 
\put(70,45){\line(-1,2){12}} 
\put(58,69){\circle*{2}} 
\put(90,55){\line(-1,2){15}} 
\put(75,85){\circle*{2}} 
\put(77,87){{\scriptsize$\overline{\sigma_{g-1}}$}} 
\put(110,65){\vector(1,0){20}} 
\put(104,56.5){{\scriptsize$\overline{\delta_C}\!=\!\overline{\tau_g}$}} 
\put(132,63){{\scriptsize$\overline{C}$}} 
\put(90,48){{\scriptsize$\overline{\tau_{g-1}}$}} 

\put(40,30){\line(2,-1){70}} 
\put(110,-5){\circle*{2}} 
\put(70,15){\circle*{2}} 
\put(90,5){\circle*{2}} 
\put(110,-5){\line(-1,-2){10}} 
\put(100,-25){\circle*{2}} 
\put(103,0){{\scriptsize${\delta_C}\!=\!{\tau_g}$}} 
\put(132,-3){{\scriptsize${C}$}} 
\put(91,6){{\scriptsize${\tau_{g-1}}$}} 
\put(70,15){\line(-1,-2){12}} 
\put(59,-9){\circle*{2}} 
\put(90,5){\line(-1,-2){15}} 
\put(75,-25){\circle*{2}} 
\put(110,-5){\vector(1,0){20}} 
\put(77,-25){{\scriptsize${\sigma_{g-1}}$}} 
\put(102,-25){{\scriptsize${\sigma_{g}}$}} 


\put(-10,30){\circle*{2}}
\put(5,30){\line(0,-1){15}}
\put(5,30){\circle*{2}}
\put(5,15){\circle*{2}}
\put(25,30){\line(0,-1){20}}
\put(25,30){\circle*{2}}
\put(25,10){\circle*{2}}

\put(-20,24){{\scriptsize ${\bf 1}\!=\!\sigma_0$}}
\put(6.5,10){{\scriptsize$\sigma_1$}}
\put(26.5,7){{\scriptsize$\sigma_q$}}
\put(4,33){{\scriptsize$\tau_1$}}
\put(24,33){{\scriptsize$\tau_q$}}
\end{picture}
$$
\caption{The minimal real resolution graph $\Gamma$ of the valuation $\{\nu_C\}$.}
\label{fig1}
\end{figure}
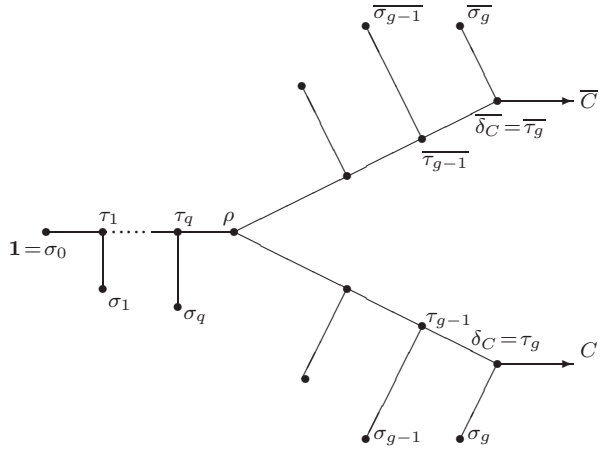

The (dual) resolution graph of $\pi$ looks, in general,
like on Figure~\ref{fig1}. The vertex $\delta_C$ corresponds to the component intersecting the strict transform of the curve $C$, the verticies $\sigma_i$ and
$\overline{\sigma_i}$, $i=0, 1, \cdots, g$, are {\em the dead ends} of the graph,
the verticies $\tau_i$ and
$\overline{\tau_i}$, $i=1, 2,\cdots, g$, are {\em the rupture points},
the vertex $\rho$ is {\em the splitting point} between
the resolutions of $C$ and of $\overline{C}$. The vertex
$\rho$ (the splitting point) may coincide with one of the rupture points $\tau_q$, $1\le q\le g$,
or with the initial point $\sigma_0$.
There are two options for the vertex $\delta_C$: either it coincides with the
rupture point $\tau_g$ (as on Figure~\ref{fig1}; this happens if the resolutions of $C$ and $\overline{C}$ split not later than each of these curves is resolved), or with the splitting point $\rho$ if, at the moment when each of the curves is resolved, their resolutions
do not split yet. In the latter case the end
of the resolution graph $\Gamma$ looks like on Figure~\ref{fig2}.

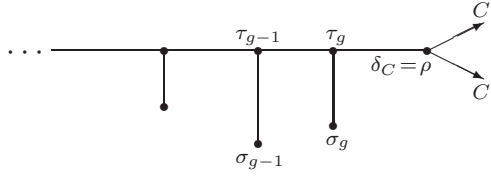
\begin{figure}
$$
\unitlength=0.50mm
\begin{picture}(160.00,40.00)(0,5)
\thinlines
\put(-10,30){\line(1,0){100}}
\put(-22,27.5){$\cdots$}
\put(65,30){\circle*{2}}
\put(20,30){\circle*{2}}
\put(45,30){\circle*{2}}
\put(90,30){\circle*{2}}
\put(90,30){\vector(2,1){15}}
\put(90,30){\vector(2,-1){15}}
\put(20,30){\line(0,-1){15}}
\put(20,15){\circle*{2}}
\put(45,30){\line(0,-1){25}}
\put(45,5){\circle*{2}}
\put(65,30){\line(0,-1){20}}
\put(65,10){\circle*{2}}

\put(39,33){{\scriptsize$\tau_{g-1}$}} 
\put(63,33){{\scriptsize$\tau_{g}$}} 
\put(75,24.5){{\scriptsize$\delta_{C}\!=\!\rho$}} 
\put(102,17){{\scriptsize${C}$}} 
\put(102,39){{\scriptsize$\overline{C}$}} 
\put(39,0){{\scriptsize$\sigma_{g-1}$}} 
\put(63,5){{\scriptsize$\sigma_{g}$}} 

\end{picture}
$$
\caption{The case when the resolutions of the curves $C$ and $\overline{C}$ split after each of the curves is resolved.}
\label{fig2}
\end{figure}
The complex conjugation acts on the resolution graph $\Gamma$ (keeping fixed the
part before the splitting point and exchanging the parts after it). The quotient of $\Gamma$ by the complex conjugation looks like on
Figure~\ref{fig3} and is a resolution graph
of the curve $C$ (the minimal one if
$\sigma_C=\tau_g$). (If $\sigma_C\ne\tau_g$,
the end of this graph looks like on
Figure~\ref{fig2}.)
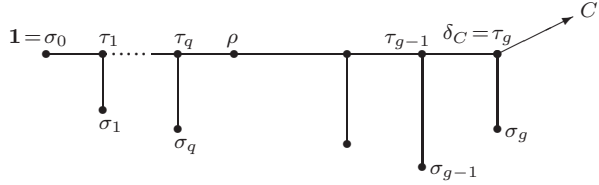
\begin{figure}[h]
$$
\unitlength=0.50mm
\begin{picture}(160.00,40.00)(0,5)
\thinlines


\put(-10,30){\line(1,0){16}}
\put(8,30){\circle*{0.5}}
\put(10,30){\circle*{0.5}}
\put(12,30){\circle*{0.5}}
\put(14,30){\circle*{0.5}}
\put(16,30){\circle*{0.5}}
\put(18,30){\line(1,0){92}}

\put(110,30){\circle*{2}}
\put(-10,30){\circle*{2}}
\put(5,30){\line(0,-1){15}}
\put(5,30){\circle*{2}}
\put(5,15){\circle*{2}}
\put(25,30){\line(0,-1){20}}
\put(25,30){\circle*{2}}
\put(25,10){\circle*{2}}

\put(40,30){\circle*{2}} 
\put(38,33){{\scriptsize$\rho$}} 
\put(70,30){\circle*{2}} 
\put(90,30){\circle*{2}} 
\put(110,30){\circle*{2}} 

\put(70,30){\line(0,-1){24}}
\put(90,30){\line(0,-1){30}}
\put(110,30){\line(0,-1){20}}
\put(70,6){\circle*{2}} 
\put(90,0){\circle*{2}} 
\put(110,10){\circle*{2}} 

\put(110,30){\vector(2,1){20}} 

\put(-20,33){{\scriptsize ${\bf 1}\!=\!\sigma_0$}}
\put(4,33){{\scriptsize $\tau_1$}}
\put(24,33){{\scriptsize $\tau_q$}}
\put(80,33){{\scriptsize $\tau_{g-1}$}}
\put(95.5,34){{\scriptsize $\delta_C\!=\!\tau_g$}}
\put(132,40){{\scriptsize $C$}}

\put(4,10){{\scriptsize $\sigma_1$}}
\put(24,5){{\scriptsize $\sigma_q$}}
\put(92,-2){{\scriptsize $\sigma_{g-1}$}}
\put(112,8){{\scriptsize $\sigma_g$}}

\end{picture}
$$
\caption{The quotient of the resolution graph $\Gamma$ by
the complex conjugation.}
\label{fig3}
\end{figure}

Let the exceptional divisor $D$ be the union of its irreducible componets $E_\sigma$, $\sigma\in\Gamma$.
(Each $E_\sigma$ is isomorphic to the complex projective line.) One has the natural involution of complex conjugation
on $\Gamma$. Let $(E_\sigma\circ E_\delta)$ be the
intersection matrix of the components of the exceptional
divisor. (The self-intersection number $E_\sigma\circ E_\sigma$ is a negative integer and, for $\delta\ne\sigma$,
$E_\sigma\circ E_\delta$ is either $1$ (if $E_\sigma$ and  $E_\delta$ intersect) or $0$ otherwise.) Let
$(m_{\sigma\delta}):=-(E_\sigma\circ E_\delta)^{-1}$.
(The entries $m_{\sigma\delta}$ are positive integers.)
Let, as above, the strict transform of the curve $C$ intersect the
exceptional divisor $D$ at a point of the component
$E_{\delta_C}$ and let $m_{\sigma}:=m_{\sigma\delta_C}$.

Let $M_\sigma\in\Z_{>0}$ be defined in the following way.
If $E_{\overline{\sigma}}=E_\sigma$, one puts $M_\sigma=m_\sigma$; if
$E_{\overline{\sigma}}\ne E_\sigma$,
one puts $M_\sigma:=m_\sigma+m_{\overline\sigma}$.


Let $S_{C}$ be the usual semigroup of the complex curve $C$, i.\,e.,
$S_{C} = \{\nu_{C}(g) : g\in \OO_{\C^2,0}\}$. The set of multiplicities
$\{m_{\sigma_0}, \ldots, m_{\sigma_g}\}$ is
the minimal set of
generators of $S_{C}$. Let $e_i :=\gcd \{\seq{m}{\sigma_0}{\sigma_i}\}$ for
$i=0,1,\ldots, g$ and let $N_i = e_{i-1}/e_i$ for $i=1,2,\ldots,g$. (One has $m_{\tau_i}=N_i m_{\sigma_i}$
for $i=1,2,\ldots, g$.)

For $\delta\in \Gamma$, let $\pi_{\delta}: (X_{\delta},D_{\delta})\to (\C^2,0)$ be
the minimal modification of $(\C^2,0)$ such that $E_{\delta}\subset D_{\delta}$. In
particular, $E_{\delta}$ is the last exceptional component appearing in $X_{\delta}$
and is produced by blow-up at a point $p_{\delta}$ of a previous component.
For $g\in \OO_{\C^2,0}$, we will denote by
$e_{\delta}(g)$ the multiplicity of the strict transform of the curve $\{g=0\}$ at
the point $p_{\delta}$. Notice that $e_\delta(g)$ coincides with the intersection
multiplicity of the strict transform of $\{g=0\}$ and $E_{\delta}$ in the surface
$X_{\delta}$.

\begin{lemma}\label{lemma:generation}
 The integers $M_{\sigma_0}$, $M_{\sigma_1}$, \dots, $M_{\sigma_g}$ generate the semigroup $S_C^{\R}$ of values of the
 valuation $\nu_C$ on $\EE_{\R^2,0}$. One has
 $(N_i-1)M_{\sigma_i}\notin \langle M_{\sigma_0}, \ldots, M_{\sigma_{i-1}}\rangle$,
 $N_i M_{\sigma_i}\in \langle M_{\sigma_0}, \ldots, M_{\sigma_{i-1}}\rangle$ and
$N_i M_{\sigma_i}< M_{\sigma_{i+1}}$.
(In particular, $M_{\sigma_0}$, $M_{\sigma_1}$, \dots,
$M_{\sigma_g}$ is the minimal set of generators of the
semigroup $S_C^{\R}$.)
\end{lemma}

\begin{proof}
Let $g\in \EE_{\R^2,0}$ be irreducible, let $G: \{g=0\}$ be the curve defined
by $g$ and let $\w{G}$ denote the strict transform of $G$
on $X$,
Making additional blow-ups at intersection points of pairs of exceptional components
one can assume that $\w{G}$ intersects $D$ at a smooth point of the component
$E_{\delta}\subset D$. Along the proof we will use known results
relating the
position of $\delta$ in the dual graph and the corresponding value $\nu_{C}(g)$: see,
e.g., \cite{Del-Lille, Casas, Wall}.

Let $q$ be the maximal index such that $\tau_q$
preceeds (or is equal to) $\rho$. (If such $q$ inbetween 1 and $g$ does not exist (i.\,e., the resolutions of $C$ and $\b{C}$ split before $\tau_1$), we put $q=0$.)
If $\delta$ preceeds $E_\rho$, i.\,e., $E_\delta\subset D_{\rho}$, then
$\nu_{C}(g) \in \langle \seq{m}{\sigma_0}{\sigma_q}\rangle$ and so
$\nu_{C}(g) \in \langle \seq{M}{\sigma_0}{\sigma_q}\rangle$. In particular,
$m_{\rho}\in \langle \seq{M}{\sigma_0}{\sigma_q}\rangle$ because
$m_{\rho}=\nu_{C}(\varphi_\rho)$, where $\{\varphi_\rho=0\}$ is a curvette at the
divisor $E_{\rho}$.

Now, let us assume that $\delta$ is after $\rho$. In this case $g$ is the product of
two irreducible complex conjugate functions, $g= f\cdot\b{f}$, equivalently
$G = F\cup \b{F}$ with $F=\{f=0\}$ being a complex irreducible curve germ. As earlier, denote by
$E_{\delta}$ the (unique) divisor on the minimal resolution of $C$ intersecting
the strict transform $\w{F}$ of $F$, so
$E_{\b{\delta}}$ is the corresponding component for the conjugate $\b{F}$ of $F$.
The splitting point of $\b{\delta}$ and $C$ is $\rho$ and
therefore
$\nu_{C}(\b{f}) = e_{\rho}(f) m_{\rho}$. Let $[\alpha,\beta]$ denote the geodesic
in $\Gamma$ between the vertices $\alpha$ and $\beta$.

If $\delta\in [\sigma_i, \tau_{i}]$ for some $i\ge q+1$, then
$\nu_C(f) = k m_{\sigma_i}$ for some positive integer $k$. For the conjugate one has
that $\b{\delta}\in [\b{\sigma}_i, \b{\tau}_{i}]$ and so
$\nu_C(\b{f}) = \nu_{\b{C}}(f)= k m_{\b{\sigma}_i}$. Thus, in this case one has
that
$\nu_{C}(g) = \nu_{C}(f)+\nu_{C}(\b{f}) = k(m_{\sigma_i}+ m_{\b{\sigma}_i}) =
k M_{\sigma_i}$.

In particular for $m_{\tau_i}$, it is known that $m_{\tau_i}= N_i
m_{\sigma_i}$ and so $M_{\tau_i} = N_i M_{\sigma_i}$. Moreover, if $\{\varphi=0\}$
is a curvette at $E_{\sigma_i}$ there is
$e_{\sigma_i}(\varphi) = e_{\tau_{i-1}}(\varphi)=1$ and therefore
\begin{equation}\label{eqn:rho}
e_\rho(\varphi) m_\rho =
m_\rho e_{\tau_{q}}(\varphi) = m_\rho
\frac{e_{\tau_q}(\varphi)}{e_{\tau_{q+1}}(\varphi)}\cdot\ldots\cdot
\frac{e_{\tau_{i-2}}(\varphi)}{e_{\tau_{i-1}}(\phi)} e_{\tau_{i-1}}(\varphi) =
N_{q+1}\cdot\ldots\cdot N_{i-1} m_\rho\;.
\end{equation}
The last equality is due to the fact that
$e_{\tau_{j-1}}(\varphi)/e_{\tau_{j}}(\varphi) = e_{j-1}/e_j$ for $j\le i-1$.
Thus one has
$M_{\sigma_{q+1}} = m_{\sigma_{q+1}}+m_\rho$ and
\begin{equation}\label{eqM}
M_{\sigma_i} = m_{\sigma_i}+m_{\b{\sigma}_i} = m_{\sigma_i} + N_{q+1}\cdot\ldots\cdot N_{i-1}
m_\rho\ \ {\text{for }}i>q+1\,.
\end{equation}

If $\delta\in [\rho, \tau_{q+1}]$, one has $\nu_C(f)\in
\langle\seq{m}{\sigma_0}{\sigma_q}\rangle$. On the other hand,
$\b{\delta}\in [\rho, \b{\tau}_{q+1}]$ and therefore
$\nu_C(\b{f}) = e_{\rho}(f) m_\rho$. Thus,
$\nu_C(g) = \nu_C(f)+\nu_C(\b{f}) \in
\langle\seq{M}{\sigma_0}{\sigma_q}\rangle$.

Finally, let us assume that
$\delta\in [\tau_{i}, \tau_{i+1}]$ for some $i\ge q+1$.
In this case
$\nu_C(f)\in \langle\seq{m}{\sigma_0}{\sigma_{i}}\rangle$ and so
$\nu_C(f) = \sum_{j=0}^{i} k_j m_{\sigma_j}$ with $k_0\ge 0$ and
$0\le k_j < N_j$ for $j=1,\ldots, i$; moreover the integers $k_j$ are
uniquely determined by
these conditions.
On the other hand, $\b{\delta}\in [\b{\tau}_i,\b{\tau}_{i+1}]$ and then
$e_{\b{\tau}_i}(\b{f}) = e_{\tau_i}(f)\ge 1$. By the equations~(\ref{eqn:rho}), one
has
$$
e_\rho(f) = e_{\tau_q}(f) = N_{q+1}\cdot\ldots\cdot N_i e_{\tau_i}(f) \ge
N_{q+1}\cdot\ldots\cdot N_i
$$
and
therefore
$\nu_C(\b{f})= N_{q+1}\cdot\ldots\cdot N_i e_{\tau_i}(f) m_\rho\ge N_{q+1}\cdot\ldots\cdot N_i
m_\rho$. By Equation~(\ref{eqM}) one has 
\begin{align*}
\sum_{j=q+1}^i k_j m_{\b{\sigma}_{j}} & =
\sum_{j=q+1}^i k_j N_{q+1}\cdot\ldots\cdot N_{j-1} m_\rho   \le
m_\rho \sum_{j=q+1}^i N_{q+1}\cdot\ldots\cdot N_{j-1} (N_j-1) \\
&= m_\rho (N_{q+1}\cdot\ldots\cdot N_i-1)
 < (N_{q+1}\cdot\ldots\cdot N_i)m_\rho  \le \nu_{C}(\b{f}) \; .
\end{align*}
As a consequence,  for some integer $a>0$, one has
\begin{align*}
\nu_C(g) &= \nu_C(f)+\nu_C(\b{f})  = \sum_{0}^i k_j m_{\sigma_j} +e_\rho(f)m_\rho =
\\
&= \sum_{0}^q k_j m_{\sigma_j} + \sum_{q+1}^i k_j
(m_{\sigma_j}+m_{\b{\sigma}_j}) + a m_\rho \\
&= \sum_{0}^i k_j M_{\sigma_j} + a m_{\rho}\in
\langle\seq{M}{\sigma_0}{\sigma_i}\rangle
\end{align*}
and the first part of the statement is proved.

The remainning statements of the Lemma are trivial consequences of the same
properties for the set of multiplicities
$\seq{m}{\sigma_0}{\sigma_g}$ (the minimal set of generators of the
semigroup $S_{\C}$), taking into account that
$e_i = \gcd\{\seq{M}{\sigma_0}{\sigma_{i}}\}$ for $0\le i \le g$.
\end{proof}


\begin{remark}\label{rem1}
 Lemma~\ref{lemma:generation} alongside with the inequality
 implies that the ``real'' semigroup $S_C^{\R}$ coincides with
 the usual semigroup $S_{C'}$ of another plane curve singularity $(C',0)$.
 Examples: for $C$ given by $x=t^4$, $y=\alpha t^4 +t^6+t^7$
 with a generic complex $\alpha$, as the curve $C'$
 one can take $x=t^4$, $y=t^{10}+t^{11}$;for $C$ given by $x=t^4$, $y=\alpha t^6+t^7$
 with a generic complex $\alpha$, as the curve $C'$
 one can take $x=t^4$, $y=t^6+t^{19}$; for $C$ given by $x=t^4$, $y=t^6+\alpha t^7$
 with a generic complex $\alpha$, as the curve $C'$
 one can take the initial curve: $x=t^4$, $y=t^6+t^7$.
In general, if one defines $b_0:=M_{\sigma_0}$, $b_1:=M_{\sigma_1}$ and
$b_{i+1}:=M_{\sigma_{i+1}}- N_{i}M_{\sigma_i}$ for $i\ge 1$, the semigroup $S_C^{\R}$ is the
semigroup of values of the complex branch defined (among others) by
$x=t^{b_0}, y=\sum_{i\ge 1} t^{b_i}$.
\end{remark}

\begin{theorem}\label{theo:Semigroup_Poincare}
 One has
 $$
 P^S_{\nu_C}(t)=\frac{\prod_{i=1}^g(1-t^{M_{\tau_i}})}
 {\prod_{i=0}^g(1-t^{M_{\sigma_i}})}\,,
 $$
\end{theorem}

\begin{proof}
 Lemma~\ref{lemma:generation} implies that any element of the semigroup $S^{\R}_C$ 
 can be in a unique way represented as
 $k_0 M_{\sigma_0}+ k_1 M_{\sigma_1}+\ldots +
 k_g M_{\sigma_g}$ with $k_i\in \Z_{\ge0}$ for $i=0,\ldots,g$ and $k_i\le N_i$ for $i\ge 1$.
 This yields the statement
 (since $M_{\tau_i}=N_i M_{\sigma_i}$).
\end{proof}

\begin{proposition}\label{prop:dimensions}
Let an element $a\in S^{\R}_C$ 
be  less that $M_{\rho}=m_{\rho}$.
 Then one has $\dim(J(a)/J(a+1))=1$,
 $\dim(J(M_{\rho})/J(M_{\rho}+1))=2$
 (and therefore $J(M_{\rho})/J(M_{\rho}+1)\cong \C$).
\end{proposition}

\begin{proof}
 Let $a=v_C(f_1)=v_C(f_2)$ with $f_i\in\EE_{\R^2,0}$. Since $a<m_{\rho}$,
 the strict transforms of the curves $\{f_i=0\}$
 intersect the exceptional divisor $D$ only at components
 preceeding $E_{\rho}$.
 One has $v_C(f_1)=v_C(f_2)< m_{\rho}$.
 The multiplicities of the liftings
 $f_1\circ\pi$ and $f_2\circ\pi$ of the germs $f_1$ and $f_2$
 along the components $E_{\sigma}$ with $\sigma\ge\rho$ are the same. Therefore on all these components the ratio $\frac{f_1\circ\pi}{f_1\circ\pi}$ is a constant different
 from zero (and from infinity). Since its value at the components $E_{\sigma}$ and $E_{\overline{\sigma}}$
 are conjugate to each other, this ratio is real.
 Therefore $\dim(J(a)/J(a+1))=1$.

 One has $m_{\rho}=\sum\limits_{i: \sigma_i<\rho}k_i m_{\sigma_i}$ with $k_i\ge 0$.
 Let us take a real function $f_1$ such that the strict transform of the curve $\{f_1=0\}$ is the union of $k_i$
 real curvettes at the components $E_{\sigma_i}$.
 Let $\pi_{\rho}:(X_\rho,D_\rho)\to(\C^2,0)$ be the modification of
 $(\C^2,0)$ produced in the course of the resolution $\pi$
 up to the moment when the component $E_{\rho}$ is created.
 Let $z$ be an affine coordinate on $E_{\rho}$ with real
 values on the real part of $E_{\rho}$ and equal to zero
 at the intersection point of $E_{\rho}$ with the previous
 component. Let $f_2$ be a real function such that the
 strict transform of $\{f_2=0\}$ is a real curvette
 at $E_{\rho}$ at infinity. The components $E_{\sigma}$ with
 $\sigma>\rho$ are obtained after blow ups the component
 $E_{\rho}$ at non-real points $z_0$ and $\overline{z_0}$.
 The ratio $\psi=\frac{f_1\circ\pi}{f_2\circ\pi}$
 on the component $E_{\rho}$ is equal to $cz$ with
 real $c$. (It has a zero at $z=0$ and a pole at $z=\infty$.)
 Therefore its value at the point $z_0$ (from which the component intersecting the strict transform of the curve $C$ emerges) is not real. The ratio $\psi$ is equal to this
 constant on all the components from the corresponding
 connected component of the part of the resolution graph
 consisting of the verticies $\sigma>\rho$.
 Therefore
 the ratio $\frac{f_1\circ \pi}{f_2\circ \pi}$ on the curve $C$ tends to a
 non-real (non-zero) number at the origin. This implies that
 $J(m_{\rho})/J(m_{\rho}+1)\cong \C$ and
 $\dim J(m_{\rho})/J(m_{\rho}+1)=2$.
\end{proof}

\begin{lemma}\label{lemma:geodesic}
 Let $\delta$ be a vertex of the resolution graph $\Gamma$
 lying on the geodesic from $\rho$ to $\delta_C$.
 Them $M_\delta\in m_\rho+ S^{\R}_{C}$.
\end{lemma}


\begin{proof}
Let $\varphi\in \OO_{\C^2,0}$ be the equation of a curvette at the point $\delta$ and
let us assume that either  $\delta\in [\tau_i, \tau_{i+1}]$, for some $i>q+1$ or
$\delta\in [\rho, \tau_{q+1}]$.
Then, $M_{\delta} = \nu_{C}(\varphi)+\nu_{C}(\b{\varphi})$ and by the last part in
the proof of Lemma~\ref{lemma:generation}, one has that
$M_{\delta} = \sum_{0}^i k_j M_{\sigma_j} + a m_{\rho}$ for some $a>0$. Thus
$M_{\delta}-m_\rho\in S^{\R}_{C}$.
\end{proof}

\begin{proposition}
 Let $S_2$ be the set of elements $a$ of the semigroup $S_C^{\R}$
 such that $\dim(J(a)/J(a+1)=2$. Then
 $S_2=m_{\rho}+S_C^{\R}$\,.
\end{proposition}

\begin{proof}
 If $a= m_\rho+b$ with $b\in S^{\R}_{C}$, then obviously $a\in S_2$
 (since $m_\rho\in S_2$).

 Assume that $a\notin m_\rho+S^{\R}_{C}$ and $\dim(J(a)/J(a+1))=2$.
 Let $f$ be a real function with $v_C(f)=a$, and let the strict
 transform of the curve $\{f=0\}$ intersect the exceptional
 divisor $D$ at points of the subset $\Gamma'$ of $\Gamma$.
 The subset $\Gamma'$ cannot contain a vertex from the geodesic from $\rho$ to
$\delta_C$ (due to Lemma~\ref{lemma:geodesic}).
 One has
$$f = f'\cdot \prod_{i:\b{\sigma}_i\neq \sigma_i} f'_i \b{f'_i}$$
where $f'$ is real function such that the strict transform of the curve $\{f'=0\}$
intersects the exceptional divisor $D$ only at points of components $E_{\delta}$ with
$\b{\delta}=\delta$, $f'_i$ is a (complex analytic) function such that the strict
transform of the curve $\{f'_i=0\}$ intersects only at points of the components
$E_\delta$ from the tail containing $E_{\sigma_i}$. One has
$\nu_C(f'_i \b{f'_i}) = k_i M_{\sigma_i}$. If $k_i\ge N_i$ then
$\nu_C(f'_i \b{f'_i})\in m_\rho+S$ (since $N_i M_{\sigma_i}=M_{\tau_i}$ and
Lemma\ref{lemma:geodesic}). Thus $k_i<N_i$.

If $g$ is another real function with $\nu_C(g)=a$ and
$g = g'\cdot \prod_{i:\b{\sigma}_i\neq \sigma_i} g'_i \b{g'_i}$, then
$\nu_C(g'_i\b{g'_i}) = \nu_C(f'_i\b{f'_i})$ (because of the uniqueness of the
representation in terms of the minimal set of generators
$\{\seq{M}{\sigma_0}{\sigma_g}\}$). This implies that
$\frac{g\circ \pi}{f\circ \pi}$ is constant on the union of the two geodesics from
$\rho$ to $\delta_C$ and $\b{\delta}_C$ ($\rho$ included). Since its values at
complex conjugate points are complex conjugate to each other, this constant is real.
Therefore the element of the extended semigroup defined by $f$ and $g$ are linear
dependent over $\R$. This contradicts the assumption that
$\dim (J(a)/J(a+1)) =2$.
\end{proof}

\begin{theorem}\label{theorem2}
One has:
\begin{align}
P_{\nu_C}(t) & = P^S_{\nu_C}(t)(1+t^{m_\rho}) = P^S_{\nu_C}(t)
\frac{1-t^{2m_\rho}}{1-t^{m_\rho}}\,; \label{eqn:dim-Poincare}\\
P^{\R}_{\nu_C}(t) &= P^S_{\nu_C}(t)(1-t^{m\rho})\nonumber\; .
\end{align}
\end{theorem}

\begin{proof}
One has
$$
P_{\nu_C}(t) = P^S_{\nu_C}(t) + \sum_{a\in S_2} t^a = P^S_{\nu_C}(t) +\sum_{v\in S}t^{m_\rho +v} =
P^S_{\nu_C}(t) + t^{m_\rho}P^S_{\nu_C}(t)\;;
$$
$$
P^{\R}_{\nu_C}(t) = P^S_{\nu_C}(t)- \sum_{a\in S_2} t^a = P^S_{\nu_C}(t) - \sum_{v\in S}t^{m_\rho +v} =
P^S_{\nu_C}(t) - t^{m_\rho}P^S_{\nu_C}(t)\; .
$$
\end{proof}

\begin{remark}
Assume that the group $\Z_2$ of order $2$ acts on $(\C^2,0)$ and
$(C',0)\subset (\C^2,0)$ is a complex curve such that the dual graph of its
equivariant resolution (together with the ages of the verticies) coincides with the
resolution graph $\Gamma$. In \cite{MMJ}, there was defined the equivariant
Poincar\'e series $P^{\Z_2}_{C'}(t)$ as an element of
$R_1(\Z_2)[[t]]$ where $R_1(\Z_2)= \Z[\sigma]/(\sigma^2-1)$ is the ring of
representations of $\Z_2$. One has its reduction under the dimensional homomorphism
$red: R_1(\Z_2)\to \Z$ as an element
$red\, P^{\Z_2}_{C'}(t)\in \Z[[t]]$. Equation \ref{eqn:dim-Poincare} and the equation of
\cite[Theorem 2]{MMJ} imply that
$P_{\nu_C}(t)=red\, P^{\Z_2}_{C'}(t)$. We have no independent explanation of this
coincidence.
\end{remark}

\begin{remark}
 As it was indicated in Remark~\ref{rem1}, the semigroup $S_C^{\R}$ is the
semigroup of values of a complex plane curve $C'$. Therefore its elements are
symmetric with respect to the conductor $c$ of the semigroup in the sense that
$a\in S \iff c-1-a\notin S$. One has one more symmetry concerning the coefficients of
the series $P_{\nu_C}(t)$. This symmetry is with respect to $c+m_\rho$:
$$
\dim \left({J(a)}\left/{J(a+1)}\right.\right) + \dim\left( {J(c+m_\rho-1-a)}\left/{J(c+m_\rho-a)}\right.\right) = 2\;
$$
(cf.~\cite{CDK}).
(Thus if one of these dimensions is  equal to zero, the other one is equal to $2$,
and if one them is equal to $1$, then the other one is equal to $1$ as well).
\end{remark}


\begin{remark}
 From Remark~\ref{rem1}, one can easily see that the semigroup Poincar\'e series
$P^S_{\nu_C}(t)$ does not determine the real topology of the curve $C$, i.\,e., the topology of the curve
 $C\cup \overline{C}$.
 On the other hand each of the series $P_{\nu_C}(t)$ and $P^{\R}_{\nu_C}(t)$
 permit to determine not only the semigroup Poincar\'e series $P^S_{\nu_C}(t)$, but also the value $m_{\rho}$
 and thus the splitting point $\rho$. Therefore each of these series determines the real topology of the curve $C$.
\end{remark}

Addresses:

A. Campillo and F. Delgado:
IMUVA (Instituto de Investigaci\'on en
Matem\'aticas), Universidad de Valladolid,
Paseo de Bel\'en, 7, 47011 Valladolid, Spain.
\newline E-mail: campillo\symbol{'100}agt.uva.es, fdelgado\symbol{'100}uva.es

S.M. Gusein-Zade:
Moscow State University, Faculty of Mathematics and Mechanics, Moscow Center for
Fundamental and Applied Mathematics, Moscow, Leninskie Gory 1, GSP-1, 119991, Russia.\\
\& National Research University ``Higher School of Economics'',
Usacheva street 6, Moscow, 119048, Russia.
\newline E-mail: sabir\symbol{'100}mccme.ru

\end{document}